\documentclass[12pt]{article}
\usepackage{amssymb}
\usepackage{amssymb,amsthm}
\usepackage{amsmath}
\usepackage[english]{babel}
\usepackage{cite}
\usepackage{float}
\usepackage [mathscr] {eucal}
\usepackage[]{graphicx}
\usepackage{sectsty}
\usepackage{caption}
\usepackage{subcaption}
\usepackage{lipsum}
\usepackage{hyperref}
\usepackage{xcolor}
\usepackage{lineno}
\hypersetup{
colorlinks = true,
linkcolor = blue,
citecolor = blue
}
\usepackage{hyperref}
\usepackage{xcolor}
\def\bc{\begin{center}}
\def\ec{\end{center}}

\def\s2c{\vskip 2cm}
\begin{document}
\renewcommand{\theequation}{\thesection.\arabic{equation}}
\numberwithin{equation}{section}
\def\mysection{\setcounter{equation}{0}\section}
\newtheorem{theorem}{Theorem}[section]
\newtheorem{problem}{Problem}[section]
\newtheorem{remark}[theorem]{remark}
\newtheorem{lemma}{Lemma}[section]
\theoremstyle{definition}
\newtheorem{example}[section]{Example}
\newtheorem{proposition}[theorem]{Proposition}
\newtheorem{Corollary}[theorem]{Corollary}
\newtheorem{definition}[theorem]{Definition}
\newtheorem{property}[]{Property}
\date{}

\title{ \textbf{Finite dimensional approximation to fractional stochastic integro-differential equations with non-instantaneous impulses}}
\author {\textbf{Shahin Ansari and Muslim Malik}\\
School of Mathematical and Statistical Sciences,\\
Indian Institute of Technology Mandi, Kamand 175005, India\\
\textbf{Emails}: shahinansaripz@gmail.com, muslim@iitmandi.ac.in  \\
 \\  }
 \maketitle
 \author
 
 \noindent \textbf{Abstract:} This manuscript proposes a class of fractional stochastic integro-differential equation (FSIDE) with non-instantaneous impulses in an arbitrary separable Hilbert space. We use a projection scheme of increasing sequence of finite dimensional subspaces and projection operators to define approximations. 
In order to demonstrate the existence and convergence of an approximate solution, we utilize stochastic analysis theory, fractional calculus, theory of fractional cosine family of linear operators and fixed point approach. Furthermore, we examine the convergence of Faedo-Galerkin (F-G) approximate solution to the mild solution of our given problem. Finally, a concrete example involving partial differential equation is provided to validate the main abstract results.
 \section{Introduction}
Fractional calculus was started by Newton, Leibniz and L’Hôpital in the second half of the eighteen century. A generalisation of the classical theory is fractional differential equations (FDEs) which allow the order of the involved derivatives to be any complex or real number. 
FDEs have a wide array of applications in engineering and applied science, such as control theory, electricity, aerodynamics, and mechanics. The key advantage of fractional calculus is that it is useful to represent the genetic properties and  memory of several materials and processes.  For more information, we consult the books \cite{N1,N2,N3}. In particular, Shu and Wang \cite{N4} studied the existence of unique mild solution for FDEs of order $1 < \alpha < 2$. The $\alpha$($1 < \alpha \leq 2$) order fractional derivative can be seen in a number of diffusion problems utilised in engineering and physical applications, such as the mechanism of super diffusion \cite{N5}.\par 
 Arbitrary fluctuations or noise are universal and indispensable in man-made systems as well as in nature, hence it is preferable to study the  stochastic systems in place of deterministic systems. Due to environmental noise deterministic system often fluctuates. As a result, dealing with stochastic differential equations is necessary. The differential equations which involve randomness in the mathematical description of certain phenomena are known as stochastic differential equations. Stochastic differential equations find application in many areas including finance \cite{F1}, engineering \cite{F2}, biology \cite{F3,F4} and population dynamics \cite{F5}. One important application of stochastic differential equations occurs in the modelling of problems associated with the percolation of fluid through porous/fractured structures and water catchment. One can see \cite{F6,F7,F8,F9,10} and the references listed therein, for further information on the stochastic system and its applications.\par
In the real world phenomena, the dynamics of processes that involve abrupt and discontinuous jumps are described by impulsive differential equations. In the literature, impulsive systems are comprehensively classified into two kinds; instantaneous impulsive systems (abrupt changes in the system remains active for a small portion of time) and non-instantaneous impulsive systems (the duration of such unforeseen changes persists for a finite time period). However, many practical problems can not characterize by the classical models with instantaneous impulses, for example, in dam pollution models, the principal source of dam pollution is the contaminated river that enters the dam and takes some time to reach the center section of the dam. The non-instantaneous impulses occur as a result of the gradual and continuous processes of the river water entering the dam and the dam water being absorbed afterwards. 
Finally, a new class of abstract semilinear non instantaneous impulsive differential equations was introduced by Hernandez and O’Regan in their work \cite{F10}.\par
It has been noticed that mostly differential and integral equations are not solvable or finding an analytic solution of such problems is difficult. Due to this complexity, a lot of focus has been placed on finding effective methods for obtaining numerical or approximate solutions of such problems. There are various conventional techniques for solving non linear and linear integro-differential equations such as Galerkin method, Chebyshev collocation method, Runge-Kutta method, rationalized Haar functional method, Adomian decomposition method, homotopy perturbation method and others. The F-G approach is among the most efficient technique to find approximate solution of a differential equation in abstract space. This method can be utilised to provide solutions of possibly weaker equations within a variational formulation \cite{F11}.\par
In 1978, by using the Heinz and Wahl's \cite{F12} existence result, Bazley \cite{F13,F14} proved the approximation of solutions to a semilinear wave equation. Then, using the premise that the nonlinear function must be defined on the whole Hilbert space $\mathcal{H}$, Goethel \cite{F15} has demonstrated the similar outcomes for a functional Cauchy problem. Miletta \cite{F16} followed Bazley's idea \cite{F13,F14} and extended the results of Goethel \cite{F15} by relaxing the assumptions imposed by Goethel \cite{F15}. After that, Bahuguna and Srivastava \cite{F17} successfully applied the F-G method to the integro-differential equation. In 2010, Balasubramaniam et al. \cite{F18} performed a novel work on F-G approximate solution of a semilinear stochastic integro-differential equation in Hilbert space.
 In \cite{F19}, Muslim established the F-G approximate results to fractional integro-differential equation with deviated arguments. Recently, Chaudhary et al. \cite{F20} used Caputo fractional derivative
and discussed some results for the F-G approximate solution to a stochastic system.
 Muslim and Agarwal \cite{F21} have worked on the F-G approximations of solutions to impulsive functional differential equations. Other works on F-G approximation to FDEs are published at \cite{F23,F24}.\par 
The study of stochastic processes involving jumps, and more specifically Levy processes, naturally leads to the development of integral differential equations. Integro-differential equations and non-linear FDEs have garnered significant interest as a research area owing to their diverse applications in traffic modeling, earthquake oscillations, and geometry. The major goal of our research is to use stochastic integro-differential equations in infinite dimensions, Caputo fractional derivative, fixed point technique, sine and cosine family of operators and, in particular, F-G approximation technique to evaluate the approximate solution for our proposed system. From the literature survey, we come to a conclusion that a number of excellent manuscripts have been published on the F-G approximation of solutions for instantaneous impulsive differential equations but there doesn't exists a single research work on approximation of solutions to non-instantaneous impulsive fractional stochastic differential system. Thus, through the work commented above and in order to make a contribution with new consequences, we are encouraged to propose a study for the existence, uniqueness and convergence for a class of approximate solutions of the non-instantaneous impulsive FSIDE of order (1,2], which is of the form 
 \begin{align}\label{main1}
 ^{c}\mathbf{D}^{\alpha}_sy(s) &= \mathcal{A}y(s)+ \mathcal{K}(s,y(s))+ \int_0^s a(s-\varsigma) \mathcal{N}(\varsigma,y(\varsigma)) d{W}(\varsigma),  \quad s \in (\varsigma_k,s_{k+1}], \nonumber \\
 &\qquad\qquad\qquad\qquad\qquad k = 0,1,2,\cdots,q, \nonumber\\
y(s) &= h_{k}^1(s,y(s_k^-)), \quad s \in (s_k,\varsigma_k],\ k = 1,2,\cdots,q, \nonumber \\
y^{\prime}(s) &= h_{k}^2(s,y(s_k^-)), \quad s \in (s_k,\varsigma_k],\ k = 1,2,\cdots,q, \\
y(0) &= y_0, \qquad y^{\prime}(0) = z_{0} \nonumber 
\end{align}
 where, the state $y(\cdot)$ is an $\mathcal{H}$-valued stochastic process; $^{c}\mathbf{D}^{\alpha}_s$ denotes the Caputo fractional derivative of order $ \alpha $; $\mathcal{A}$ is the infinitesimal generator of a strongly continuous $ \alpha $-order cosine family $\{C_{\alpha}(s) : s \geq 0\}$ on $\mathcal{H}$; $\mathcal{K}$ and $\mathcal{N}$ are non-linear functions satisfying few conditions that will be specified later; the map $a \in L_{\text{loc}}^{2p}(0,\infty)$ for some $1 \leq p < \infty$ and and $\varsigma_k, s_k$ are some pre-fixed numbers which satisfy the relation $0 = \varsigma_0 = s_0 < s_1 < \varsigma_1 < s_2 < \cdots < \varsigma_k < s_{k+1} = T$. $\{ W(s) \}_{s \geq 0}$ is the Wiener process or Brownian motion in a separable Hilbert space $\mathcal{Y}$.\vspace{2mm}
 
The following are the major contributions to our manuscript:\vspace{1mm}\\
$\bullet$ To the best of author's knowledge, this is the first work that deals the finite dimensional approximation to a class of nonlinear non-instantaneous impulsive FSIDE of order $1 < \alpha \leq 2$ in abstract space.\vspace{1.25mm}\\
$\bullet$ The F-G approximation results for the non-instantaneous impulsive system are discussed.\vspace{1.25mm}\\
$\bullet$ The conversion of a stochastic partial differential equation with impulsive conditions to an abstract stochastic differential system in infinite dimensional space is shown in the example section.\vspace{3mm}\\
The remainder of the paper is structured as follows:\vspace{0.25mm}\\
Section 2 compiles preliminary results, relevant notations and definitions that will be utilized throughout the paper. In section 3, by the virtue of stochastic theory and fixed point technique, we examine the existence of unique solution for every approximate integral equation. In section 4, we establish the convergence of the approximate solutions. Section 5 focus on the main results of this manuscript pertaining to the F-G approximate solution and convergence of such an approximation. At last, in section 6, an example is
provided by applying the obtained results to a partial differential equation.

\section{Preliminaries}
 Some definitions and mathematical techniques that will be required to prove main results of this manuscript are presented in this section.\par
Let $L(\mathcal{Y};\mathcal{H})$ be the space of all bounded linear operators from $\mathcal{Y}$ into $\mathcal{H}$. Notation $\lVert \cdot \rVert$ represents the norms of the spaces $\mathcal{H}$, $\mathcal{Y}$ and $L(\mathcal{Y};\mathcal{H})$. Throughout this paper $B_R(\mathcal{Z})$ denotes the closed ball with center at zero and radius $R > 0$ in a Banach space $\mathcal{Z}$. Let $(\Omega,\mathcal{F},\{\mathcal{F}_s\}_{s \geq 0},\mathrm{P})$ be filtered complete probability space such that the filtration $\{\mathcal{F}_s\}_{s \geq 0}$ is an increasing family of right continuous sub $\sigma-$ algebra of $\mathcal{F}$. An $\mathcal{F}-$measurable  function $y(s) : \Omega \rightarrow \mathcal{H}$ is a $\mathcal{H}-$valued random variable. A stochastic process is a collection of random variables $S = \{y(s) : \Omega \rightarrow \mathcal{H}| s \in [0,T]\}$. Let $\{e_j \}_{j \geq 1}$ is a complete orthonormal basis of the space $\mathcal{Y}$ and $ \lambda_j, (j \in \mathbb{N})$ are non-negative real numbers. For arbitrary $s \geq 0$, $W$ has the expansion
\begin{equation*}
 W(s) = \sum_{j = 1}^{\infty} \sqrt{\lambda_j}  e_j \beta_j(s), \quad
\end{equation*} 
where, $\beta_j(s) = \frac{1}{\sqrt{\lambda_j}}\langle W(s),e_j \rangle_{\mathcal{Y}}$ is a sequence of Brownian motions which are mutually independent and real valued. The reader may refer \cite{F8} for more details of this section. Moreover, the linear non-negative covariance operator $Q \in L(\mathcal{Y};\mathcal{Y})$ is symmetric and defined by $Qe_j = \lambda_j e_j$ and of finite trace $Tr(Q) = \sum\limits_{j = 1}^{\infty} \lambda_j < \infty$.   
 Then the preceding $\mathcal{Y}$-valued Weiner process $W(s)$ is called $Q-$Wiener process. Let $\mathcal{Y}_0 = Q^{\frac{1}{2}}(\mathcal{Y})$ be the closed subspace of $\mathcal{Y}$ and the space of all Hilbert Schmidt operators represented by $L_2^0 = L_2(\mathcal{Y}_0;\mathcal{H})$, is a separable Hilbert space endowed with the norm $\lVert \psi \rVert_{L_{2}^0}^2 = Tr((\psi Q^{\frac{1}{2}})(\psi Q^{\frac{1}{2}})^{\ast})$, for $\psi \in L_2^0$. Moreover, for any $ \chi \in L(\mathcal{Y};\mathcal{H})$ this norm reduces to $\lVert \chi \rVert_{L_2^0}^2 = Tr(\chi Q {\chi}^{\ast}) = \sum_{j = 1}^{\infty} \lVert \sqrt{\lambda_j} \chi e_j \rVert^2$. \\
For further development, we make use of the following results:\\
\begin{lemma}\label{lemma}
\cite{6} Let $\zeta(\cdot)$ be a measurable $L_2^0$-valued process satisfying $\mathbb{E} \int_0^T \lVert \zeta(\varsigma) \rVert^2_{L_2^0} d{\varsigma} < \infty$ then
\begin{align*}
\mathbb{E} \big[\sup_{0 \leq s \leq T} \lVert \int_0^s \zeta(\varsigma) dW(\varsigma) \rVert^2 \big]  &\leq 4 \mathbb{E} \lVert \int_0^T \zeta(\varsigma) dW(\varsigma) \rVert^2\\
& \leq 4 Tr(Q) \int_0^T \mathbb{E} \lVert \zeta(\varsigma) \rVert_{L_2^0}^2 d\varsigma.
\end{align*}

\end{lemma}
\begin{definition}
The $\alpha( > 0)$ order Riemann-Liouville fractional integral is defined by:
\begin{equation*}
J_s^{\alpha} x(s) = \frac{1}{\Gamma{\alpha}} \int_0^s (s-\varsigma)^{\alpha-1}x(\varsigma)d{\varsigma}, \quad s > 0
\end{equation*}
where, $\Gamma(\cdot)$ is the gamma function and $x(s) \in L^1([0,T];\mathcal{H})$.
\end{definition}
\begin{definition}
The $\alpha \in (1,2]$ order Riemann-Liouville fractional derivative of a function $x(s) \in L^1([0,T];\mathcal{H})$
is defined by:
\begin{equation*}
D^{\alpha}_s x(s) = \frac{d^2}{ds^2}J_s^{2-\alpha}x(s).
\end{equation*}
\end{definition}
\begin{definition}
The $\alpha \in (1,2]$ order Caputo fractional derivative of a function $x(s) \in C^1([0,T];\mathcal{H}) \cap L^1([0,T];\mathcal{H})$ is defined by:
\begin{equation*}
^{c}{D}_s^{\alpha}x(s) = J_s^{2-\alpha} \frac{d^2}{ds^2}x(s).
\end{equation*}
\end{definition}
Consider the following $\alpha \in (1,2]$ order differential problem of the form:
\begin{align}\label{2.1}
^{c}{D}_s^{\alpha}x(s) &= \mathcal{A}x(s)\nonumber \\
x(0) &= z_0, \quad x^{\prime}(0) = 0
\end{align}
where, $\mathcal{A} : D(\mathcal{A}) \subset \mathcal{H} \rightarrow \mathcal{H}$ is densely defined closed linear operator on a separable Hilbert space $\mathcal{H}$.
\begin{definition}
\cite{1} A one parameter family $(C_{\alpha}(s))_{s \geq 0} \subset L(\mathcal{H};\mathcal{H})$, $1 < \alpha \leq 2,$ is called a strongly continuous fractional cosine family for (\ref{2.1}) or the solution operator if the following conditions satisfy:\vspace{2mm}\\
(i) $C_{\alpha}(0) = I$, where I denotes the identity operator and $C_{\alpha}(s)$ is strongly continuous;\\
(ii) $C_{\alpha}(s)D(\mathcal{A}) \subset D(\mathcal{A})$ and $\mathcal{A}C_{\alpha}{z_0} = C_{\alpha}\mathcal{A}{z_0}$ for all $z_0 \in D(\mathcal{A})$, $s \geq 0$;\\
(iii) $C_{\alpha}(s)z_0$ is the solution of $x(s) = z_0 + \frac{1}{\Gamma \alpha} \int_0^s (s-\varsigma)^{\alpha-1}\mathcal{A}x(\varsigma)d\varsigma$ for all $z_0 \in D(\mathcal{A})$. 
\end{definition}

\begin{definition}
$S_{\alpha} : [0,\infty) \rightarrow L(\mathcal{H};\mathcal{H})$ denotes the fractional sine family associated with $C_{\alpha}$ is given by:
\begin{equation*}
S_{\alpha}(s) = \int_0^s C_{\alpha}(t)dt, \quad s \geq 0.
\end{equation*}
\end{definition}
\begin{definition}
$P_{\alpha} : [0,\infty) \rightarrow L(\mathcal{H};\mathcal{H})$ denotes the $\alpha$-order Riemann-Liouville family associated with $C_{\alpha}$ is given by
\begin{equation*}
P_{\alpha}(s) = J_s^{\alpha-1} C_{\alpha}(s).
\end{equation*}
\end{definition}
\begin{definition}
The fractional cosine family $C_{\alpha}(s)$ is said to be exponentially bounded if there exists constants $\omega \geq 0$ and $M \geq 1$ such that 
\begin{equation}\label{2.2}
\lVert C_{\alpha}(s) \rVert \leq M e^{\omega s}, \quad s \geq 0.
\end{equation}
\end{definition}
If problem (\ref{2.1}) has a solution operator $C_{\alpha}(s)$ satisfying (\ref{2.2}) then an operator $\mathcal{A}$ is referred to as belongs to $C^{\alpha}(\mathcal{H};M,\omega)$. Let $1 < \alpha \leq 2$ and $\mathcal{A} \in C^{\alpha}(\mathcal{H}; M,\omega)$, therefore by Theorem (3.3) in \cite{1}, -$\mathcal{A}$ be infinitesimal generator of an analytic semigroup and hence we can define a closed linear operator of fractional power $\mathcal{A}^{\beta}$ for $\beta \in (0,1]$ on its domain $D(\mathcal{A}^{\beta})$ such that $\overline{D(\mathcal{A}^{\beta})} = \mathcal{H}$. $\mathcal{H}_{\beta}$ stands for $D(\mathcal{A}^{\beta})$ forms a Banach space endowed with the norm $\lVert y \rVert_{\beta} = \lVert \mathcal{A}^{\beta}y \rVert$, for any $y \in D(\mathcal{A}^{\beta})$. For more knowledge about fractional power of operators and analytic semigroup, inspect \cite{3,S1}.\par
Let $L_2(\Omega,\mathcal{F},P;\mathcal{H}_{\beta}) \equiv L_2(\Omega;\mathcal{H}_{\beta})$ denotes the space of all square integrable, strongly measurable $\mathcal{H}_{\beta}$-valued random variables, which is a Banach space endowed with the norm 
\begin{equation*}
\lVert z \rVert_{L_2(\Omega;\mathcal{H}_{\beta})} = \big(\int_{\Omega} \lVert z(s) \rVert^2_{\beta}\ dP\big)^{\frac{1}{2}}=(\mathbb{E} \lVert z(s) \rVert^2_{{\beta}})^\frac{1}{2}
\end{equation*}
where, $\mathbb{E}$ denotes integration with respect to probability measure P i.e. $\mathbb{E}z = \int_{\Omega}z dP$.  \\
In order to deal the impulses, we consider the space of piecewise continuous functions \\$PC([0,T];L_2(\Omega;\mathcal{H}))$ formed by $L_2(\Omega;\mathcal{H})$-valued stochastic processes $\{y(s) : s \in [0,T] \}$ such that $y(s)$ is continuous at $s \neq s_k$, $y(s)$ is left continuous on $[0,T]$ and $\lim_{s \downarrow s_k}y(s)$ exists for all $k = 1, 2,\cdots,q$ which forms a Banach space endowed with the sup norm $\lVert y \rVert_{PC} = \bigg(\sup\limits_{ 0 \leq s \leq T} \mathbb{E} \lVert y(s) \rVert^2 \bigg)^{\frac{1}{2}}$.\\
Also, $PC_T^{\beta}=PC([0,T];L_2(\Omega;\mathcal{H}_{\beta}))$ be the Banach space formed by $L_2(\Omega;\mathcal{H}_{\beta})$-valued stochastic processes $\{y(s) : s \in [0,T] \}$ such that $y(s)$ is continuous at $s \neq s_k$, $y(s)$ is left continuous on $[0,T]$ and $\lim_{s \downarrow s_k}y(s)$ exists for all $k = 1, 2,\cdots,q$ which forms a Banach space with the sup norm $\lVert y \rVert_{PC, \beta} = \bigg(\sup\limits_{0 \leq s \leq T} \mathbb{E} \lVert \mathcal{A}^{\beta}y(s) \rVert^2 \bigg)^{\frac{1}{2}}$.\\
We impose the following assumptions in order to prove main results.\vspace{2mm}\\
\textbf{(A1):} $\mathcal{A} : D(\mathcal{A}) \subset \mathcal{H} \rightarrow \mathcal{H}$ is self-adjoint, positive definite, closed linear operator  such that $\overline{D(\mathcal{A})} = \mathcal{H}$. Suppose that, the pure point spectrum of $\mathcal{A}$ is\\
\begin{equation*}
0 < \lambda_{0} \leq \lambda_{1} \leq \lambda_{2} \leq \cdots \leq \lambda_{n} \leq \cdots
\end{equation*}
where $\lambda_{n} \rightarrow \infty$ as $n \rightarrow \infty$ and an orthonormal system of eigenfunctions $\psi_{j}$ is complete, which is corresponding to $\lambda_j$ i.e.
\begin{equation*}
\mathcal{A} \psi_{j} = \lambda_{j}\psi_{j} \quad \text{and} \quad \langle  \psi_{i} , \psi_{j} \rangle = \delta_{ij},
\end{equation*}
 where $\delta_{ij}$ denotes the Kronecker Delta function.\\
 Also
 \begin{equation}
 \lVert C_{\alpha}(s) \rVert \leq M \nonumber.
 \end{equation}\vspace{2mm}
\textbf{(A2):} The function $\mathcal{K} : [0,T] \times \mathcal{H}_{\beta} \rightarrow \mathcal{H}$ is nonlinear and fulfils the requirement
\begin{equation*}
\mathbb{E} \lVert \mathcal{K}(s,y) - \mathcal{K}(s,x) \rVert^2 \leq L_\mathcal{K} \mathbb{E} \lVert y-x \rVert^2_{\beta}
\end{equation*}
and $\mathbb{E}\lVert \mathcal{K}(s,y) \rVert^2 \leq L_\mathcal{K}^{\prime}, \  s \in [0,T], \ y \in B_R(\mathcal{H}_{\beta}),$
where $B_R(\mathcal{H}_{\beta}) = \{y \in \mathcal{H}_{\beta} : \lVert y \rVert_{\beta} \leq R \}$ and $L_\mathcal{K},L_\mathcal{K}^{\prime}$(depending on R) are positive constants.\vspace{2mm} \\
\textbf{(A3):} The nonlinear map $\mathcal{N} : [0,T] \times \mathcal{H}_{\beta} \rightarrow L_2^0$ and there exist non-negative functions $L_\mathcal{N},L_\mathcal{N}^{\prime}(\text{depending on R}) \in L_{loc}^q(0,\infty)$, where $1<q<\infty, \ \frac{1}{p}+\frac{1}{q}=1$ such that
\begin{equation*}
\mathbb{E}\lVert \mathcal{N}(s,y)-\mathcal{N}(s,x) \rVert^2_{L_2^0} \leq L_\mathcal{N}(s) \mathbb{E}\lVert y-x \rVert^2_{\beta}
\end{equation*}
and 
\begin{equation*}
\mathbb{E}\lVert \mathcal{N}(s,y) \rVert^2_{L_2^0} \ \leq \ L_\mathcal{N}^{\prime}(s), \ s \in [0,T], \ y \in B_R(\mathcal{H}_{\beta}).\vspace{2mm}\\
\end{equation*}
\textbf{(A4):} The functions $h_k^i : (s_k,\varsigma_k] \times \mathcal{H}_{\beta} \rightarrow \mathcal{H}_{\beta}$, $i =1,2, \ k = 1,2,\cdots,q$ are continuous and there exist constants $D_{h_k^i},C_{h_k^i}(\text{depends on R}) > 0$, $k = 1,2,\cdots,q, \ i =1,2$ such that
\begin{equation*}
\mathbb{E}\lVert h_{k}^i(s,y) - h_{k}^i(s,x) \rVert^2_{\beta} \leq D_{h_k^i} \mathbb{E}\lVert y-x \rVert^2_{\beta} 
\end{equation*}
also
\begin{equation*}
\mathbb{E}\lVert h_{k}^i(s,x) \rVert^2_{\beta} \leq C_{h_k^i}, \ \forall x \in B_R(\mathcal{H}_{\beta}).
\end{equation*}
\vspace{2mm}
\\
For the notational convenience, we denote\\
$ \rho = \sup\limits_{0 \leq s \leq T} \lVert \mathcal{A}P_{\alpha}(s) \rVert; \quad D = \max\limits_{1 \leq k \leq q}\bigg\{\max\limits_{0 \leq k \leq q} Q_{k}, D_{h^1_k}\bigg\},\ k = 1,2,\cdots,q;$ \vspace{2mm}\\
$N_{0} = 4\{M^2 \lVert y_0 \rVert^2_{\beta} + M^2T^2\lVert z_0 \rVert^2_{\beta} + \lVert \mathcal{A}^{\beta-1} \rVert^2 \rho^2 s_1 (L_{\mathcal{K}}^{\prime} + 4Tr(Q)\lVert a^2 \rVert_{L^p(0,T)} \lVert {L_\mathcal{N}^{\prime}} \rVert_{L^q(0,T)})\}$;\vspace{2mm}\\ 
$N_{k}=4\{M^2 C_{h_k^1}+M^{2}T^{2}C_{h_k^2}+\lVert \mathcal{A}^{\beta-1} \rVert^2 \rho^2 s_{k+1}({L_\mathcal{K}^{\prime}}+ 4Tr(Q)\lVert a^2 \rVert_{L^p(0,T)} \lVert {L_\mathcal{N}^{\prime}}\rVert_{L^q(0,T)})\},\\ 
\qquad\qquad\qquad\qquad\qquad\qquad k = 1,2,\cdots,q$;\vspace{2mm}\\
$Q_0 = 2\{\lVert \mathcal{A}^{\beta-1} \rVert^2 \rho^2 s_1(L_\mathcal{K} + 4Tr(Q) \lVert {L_\mathcal{N}} \rVert_{L^q(0,T)} \lVert a^2 \rVert_{L^p(0,T)})\};$\vspace{2mm}\\
$Q_{k} = 4\{M^2 D_{h_k^1}+M^2T^2D_{h_k^2}+\lVert \mathcal{A}^{\beta-1} \rVert^2 \rho^2 s_{k+1}(L_{\mathcal{K}}+4Tr(Q) \lVert L_\mathcal{N} \rVert_{L^q(0,T)} \lVert a^2 \rVert_{L^p(0,T)})\}, \\ k = 1,2,\cdots,q.$\vspace{2mm}
\begin{definition}
An $\mathcal{F}_s$-adapted stochastic process $\{y(s) : s \in [0,T] \}$ is said to be a mild solution of the stochastic fractional system (\ref{main1}) if for any $s \in [0,T]$, $y(s)$ fulfils $y(0) = y_0$, $y^{\prime}(0) = z_0$ and 
\begin{eqnarray}\label{2.3}
y(s) &=& h_k^1(s,y(s_k^-)), \ s \in (s_k,\varsigma_k],  k = 1,2,\cdots,q; \nonumber \\
y^{\prime}(s) &=& h_k^2(s,y(s_k^-)), \ s \in (s_k,\varsigma_k],  k = 1,2,\cdots,q;\nonumber \\
y(s) &=& C_{\alpha}(s)y_0 + S_{\alpha}(s)z_0 + \int_0^s P_{\alpha}(s-\varsigma)[\mathcal{K}(\varsigma,y(\varsigma))\nonumber \\ 
&\quad & +\int_0^{\varsigma} a(\varsigma-r)\mathcal{N}(r,y(r)) dW(r)]d{\varsigma}, \ s \in (0,s_1];\\
y(s) &=& C_{\alpha}(s-\varsigma_k)h_{k}^1(\varsigma_k,y(s_k^-))+S_{\alpha}(s-\varsigma_k)h_{k}^2(\varsigma_k,y(s_k^-)) \nonumber \\
&\quad &+\int_{\varsigma_k}^s P_{\alpha}(s-\varsigma)[\mathcal{K}(\varsigma,y(\varsigma))+\int_0^\varsigma a(\varsigma-r)\mathcal{N}(r,y(r)) dW(r)]d{\varsigma}, s \in (\varsigma_k,s_{k+1}], \nonumber \\
& \quad &\qquad\qquad\qquad\qquad\qquad k = 1,2,\cdots,q.\nonumber
\end{eqnarray}
\end{definition}
\section{Existence of approximate solutions}
 The existence outcomes for fractional stochastic system (\ref{main1}) are proved in this section.\\
 We define $\mathcal{H}_n = \text{span}\{\psi_0, \psi_1, \psi_2,\cdots,\psi_n\} \subset \mathcal{H}$. Clearly, $\mathcal{H}_n$ is a subspace of $\mathcal{H}$ of finite dimension and let $\mathcal{P}^n : \mathcal{H} \rightarrow \mathcal{H}_n$ be the corresponding projection operator associated with the spectral set of the operator $\mathcal{A}$ for $n \in \mathbb{N}_0$.\\
We define
\begin{align*}
\mathcal{K}_{n}: [0,T] \times \mathcal{H}_{\beta} \rightarrow \mathcal{H}
\end{align*}
such that
\begin{align*}
\mathcal{K}_{n}(s,x(s)) = \mathcal{K}(s,\mathcal{P}^{n}x(s)),
\end{align*}
and
\begin{equation*}
h_{\text{k,n}}^i : (s_{k}, \varsigma_k] \times \mathcal{H}_{\beta} \rightarrow \mathcal{H}_{\beta}, \quad i = 1,2
\end{equation*}
such that
\begin{equation*}
h_{\text{k,n}}^i(s,x(s)) = h_{k}^i(s,\mathcal{P}^{n}x(s))
\end{equation*}
and 
\begin{equation*}
\mathcal{N}_{n} : [0,T] \times \mathcal{H}_{\beta} \rightarrow L_2^0
\end{equation*}
such that
\begin{equation*}
\mathcal{N}_{n}(s,x(s)) = \mathcal{N}(s,\mathcal{P}^nx(s)).
\end{equation*}
Consider a subset $B_R(PC_T^{\beta}) \subseteq PC_T^{\beta}$ defined as
\begin{equation*}
B_R(PC_T^{\beta}) = \{ z \in PC_T^{\beta} : \lVert z \rVert_{PC,\beta} \leq R \},
\end{equation*} 
where
\begin{equation*}
R^2 = \max\limits_{1 \leq k \leq q}\bigg(\max\limits_{0 \leq k \leq q} N_k, C_{h_k^1} \bigg).
\end{equation*}
The inequality of the form $(\sum_{j = 1}^n c_j)^m \leq n^{m-1} \sum_{j = 1}^n  c_j^m$ (where $c_j$ are nonnegative constants $j = 1, 2,\cdots,n$ and $n,m \in \mathbb{N}$) is helpful in producing various estimates along with the well-known Hölder and Minkowski inequalities.\\
We define the maps $\phi_n$ on $B_R(PC_T^{\beta})$, $n \in \mathbb{N}$, as follows,
\begin{equation*}
(\phi_{n}y)(s)=\begin{cases}
C_{\alpha}(s)y_0 + S_{\alpha}(s)z_0 + \int_0^s P_{\alpha}(s-\varsigma)[\mathcal{K}_n(\varsigma,y(\varsigma))\\
\qquad\qquad\qquad\qquad+\int_0^{\varsigma} a(\varsigma-r)\mathcal{N}_n(r,y(r)) dW(r)]d{\varsigma}, \ s \in (0,s_1]\\
 h_{k,n}^1(s,y(s_k^-)), \quad s \in (s_k,\varsigma_k], \ k = 1,2,\cdots,q, \\
  C_{\alpha}(s-\varsigma_k)h_{k,n}^1(\varsigma_k,y(s_k^-))+S_{\alpha}(s-\varsigma_k)h_{k,n}^2(\varsigma_k,y(s_k^-)) \\
 \qquad\quad  + \int_{\varsigma_k}^s P_{\alpha}(s-\varsigma)[\mathcal{K}_n(\varsigma,y(\varsigma)) +\int_0^\varsigma a(\varsigma-r)\mathcal{N}_n(r,y(r)) dW(r)]d\varsigma, \\
 \qquad\qquad\qquad\qquad\qquad\qquad\qquad\qquad s \in (\varsigma_k,s_{k+1}], \ k = 1,2,\cdots,q,
\end{cases}
\end{equation*}
for any $y \in B_R(PC_T^{\beta})$.
\begin{theorem}\label{thrm3.1}
Let $y_0,z_0 \in D(\mathcal{A})$ and all the hypotheses \textbf{(A1)-(A4)} along with $D < 1$ hold. Then, the system (\ref{main1}) has a unique mild solution $y_{n} \in B_R(PC_T^{\beta})$ such that $\phi_n y_n = y_n$, for each $n \in \mathbb{N}$. 
\begin{proof}
The proof is split up into two steps for ease of understanding.
\vspace{2mm}
\\
\textbf{Step 1:} First we demonstrate that $\phi_n : B_R(PC_T^{\beta}) \rightarrow B_R(PC_T^{\beta})$ is well defined before applying fixed point theorem.\\
For any $y \in B_R(PC_T^{\beta})$ and $s \in (0,s_1]$, we have
\begin{align} \label{3.1}
\mathbb{E} \lVert (\phi_n y)(s) \rVert_{\beta}^2 &\leq 4 \bigg[ \lVert C_{\alpha}(s)y_0 \rVert^2_{\beta} + \lVert S_{\alpha}(s)z_0 \rVert^2_{\beta} + \lVert \mathcal{A}^{\beta-1} \rVert^2 \int_0^s \lVert \mathcal{A}P_{\alpha}(s-\varsigma) \rVert^2 \nonumber \\
& \quad \times \bigg\{ \mathbb{E} \lVert \mathcal{K}_n(\varsigma,y(\varsigma)) \rVert^2 \nonumber + 4 Tr(Q) \int_0^T |a(\varsigma-r)|^2 \mathbb{E} \lVert \mathcal{N}_n(r,y(r)) \rVert^2_{L_2^0} dr \bigg\} d{\varsigma} \bigg] \nonumber \\
& \leq 4 \bigg[M^2 \lVert y_0 \rVert^2_{\beta} + M^2 T^2 \lVert z_0 \rVert^2_{\beta} + \lVert \mathcal{A}^{\beta-1} \rVert^2 \rho^2 s \nonumber \\
& \quad \times \bigg\{L_{\mathcal{K}}^{\prime} + 4Tr(Q) \lVert a^2 \rVert_{L^p(0,T)} \lVert {L_\mathcal{N}^{\prime}} \rVert_{L^q(0,T)} \bigg\} \bigg] \nonumber \\
&  \leq N_0.
\end{align}
Similarly, for any $s \in (s_k,\varsigma_k], \ k = 1,2,\cdots,q$, we can easily get
\begin{equation} \label{3.2}
\mathbb{E} \lVert (\phi_n y)(s) \rVert_{\beta}^2 = \mathbb{E} \lVert h_{k,n}^1(s,y(s_k^-)) \rVert_{\beta}^2 \leq C_{h_k^1}.
\end{equation}
Now, for any $s \in (\varsigma_k,s_{k+1}], \ k = 1,2,\cdots,q$ and $y \in B_R(PC^{\beta}_T)$, we have
\begin{align} \label{3.3}
\mathbb{E} \lVert (\phi_n y)(s) \rVert_{\beta}^2 &\leq 4 \bigg[ \lVert C_{\alpha}(s-\varsigma_k) \rVert^2 \mathbb{E} \lVert h_{k,n}^1(\varsigma_k,y(s_k^-)) \rVert^2_{\beta} + \lVert S_{\alpha}(s-\varsigma_k) \rVert^2   \nonumber \\
& \quad \times \mathbb{E} \lVert h_{k,n}^2(\varsigma_k,y(s_k^-)) \rVert^2_{\beta}+ \lVert \mathcal{A}^{\beta-1} \rVert^2 \int_{\varsigma_{k}}^s \lVert \mathcal{A}P_{\alpha}(s-\varsigma) \rVert^2 \bigg\{ \mathbb{E} \lVert \mathcal{K}_n(\varsigma,y(\varsigma)) \rVert^2   \nonumber \\
& \quad  + 4 Tr(Q)\int_0^T |a(\varsigma-r)|^2 \mathbb{E} \lVert \mathcal{N}_n(r,y(r)) \rVert^2_{L_2^0} dr \bigg\} d{\varsigma} \bigg]\nonumber  \\
& \leq 4 \bigg[ M^2C_{h_k^1}  + M^2 T^2 C_{h_k^2}+ \lVert \mathcal{A}^{\beta-1} \rVert^2 \rho^2 \nonumber \\
& \quad \times\bigg\{L_{\mathcal{K}}^{\prime} + 4Tr(Q) \lVert a^2 \rVert_{L^p(0,T)} \lVert {L_\mathcal{N}^{\prime}} \rVert_{L^q(0,T)} \bigg\} (s-\varsigma_k)\bigg] \nonumber \\
& \leq N_k.
\end{align}
After summarizing the above inequalities $(\ref{3.1})-(\ref{3.3})$, we get 
\begin{equation*}
\lVert (\phi_n y) \rVert_{PC,\beta} \leq R.
\end{equation*}
\vspace{2mm}
\\
\textbf{Step 2:} Now, the only thing left is to show that $\phi_n$ is a strict Banach contraction map on $B_R(PC_T^{\beta})$.\\
For any $y,x \in B_R(PC_T^{\beta})$ and $s \in (0,s_1]$,
\begin{align}\label{3.4}
\mathbb{E} \lVert (\phi_n y)(s) - (\phi_n x)(s) \rVert_{\beta}^2 &\leq 2\bigg[ \lVert \mathcal{A}^{\beta-1} \rVert^2 \int_0^s \lVert \mathcal{A}P_{\alpha}(s-\varsigma) \rVert^2 \bigg\{ \mathbb{E} \lVert \mathcal{K}_n(\varsigma,y(\varsigma))  \nonumber \\
& \quad - \mathcal{K}_n(\varsigma,x(\varsigma)) \rVert^2+ 4Tr(Q) \int_0^T |a(\varsigma-r)|^2 \nonumber\\
&\quad \times \mathbb{E} \lVert \mathcal{N}_n(r,y(r)) - \mathcal{N}_n(r,x(r)) \rVert^2_{L_2^0} dr \bigg\} d{\varsigma}\bigg] \nonumber\\
& \leq 2\bigg[ \lVert \mathcal{A}^{\beta-1} \rVert^2 \rho^2 s \bigg\{L_\mathcal{K}+4Tr(Q)\lVert a^2 \rVert_{L^p(0,T)} \lVert L_\mathcal{N} \rVert_{L^q(0,T)} \bigg\}  \nonumber \\
& \quad \times \mathbb{E} \lVert y(s)-x(s) \rVert^2_{\beta}\bigg] \nonumber\\
& \leq Q_0 \lVert y-x \rVert_{PC,\beta}^2.
\end{align}
Similarly, for any $x,y \in B_R(PC_T^{\beta})$ and $s \in (s_k,\varsigma_k], \ k = 1, 2, \cdots,q$, we have
\begin{align}\label{3.5}
\mathbb{E} \lVert (\phi_n y)(s) - (\phi_n x)(s) \rVert_{\beta}^2 &= \mathbb{E} \lVert h_{k,n}^1(s,y(s_k^-)) - h_{k,n}^1(s,x(s_k^-)) \rVert^2_{\beta} \nonumber\\
& \leq D_{h_k^1} \lVert y -x \rVert_{PC,\beta}^2.
\end{align}
Now, for any $x,y \in B_R(PC_T^{\beta})$ and $s \in (\varsigma_k,s_{k+1}], \ k = 1, 2, \cdots,q$, we have
\begin{align}
\mathbb{E} \lVert (\phi_n y)(s) - (\phi_n x)(s) \rVert_{\beta}^2 & \leq 4[I_1+I_2+I_3+I_4].
\end{align}
Where,
\begin{align}\label{3.7}
I_1 &= \mathbb{E}\lVert C_{\alpha}(s-\varsigma_k)\big[h_{k,n}^1(\varsigma_k,y(s_k^-)) - h_{k,n}^1(\varsigma_{k},x(s_k^-))\big] \rVert^2_{\beta} \nonumber \\
& \leq \lVert C_{\alpha}(s-\varsigma_k) \rVert^2 \mathbb{E}\lVert h_{k,n}^1(\varsigma_{k},y(s_k^-)) - h_{k,n}^1(\varsigma_k,x(s_k^-)) \rVert^2_{\beta} \nonumber \\
& \leq M^2 D_{h_k^1}\mathbb{E} \lVert y(s_k^-)-x(s_k^-) \rVert_{\beta}^2.\\
I_2 &= \mathbb{E}\lVert S_{\alpha}(s-\varsigma_k) \big[h_{k,n}^2(\varsigma_k,y(s_k^-)) - h_{k,n}^2(\varsigma_{k},x(s_k^-)) \big]\rVert^2_{\beta} \nonumber \\
& \leq \lVert S_{\alpha}(s-\varsigma_k) \rVert^2 \mathbb{E}\lVert h_{k,n}^2(\varsigma_k,y(s_k^-)) - h_{k,n}^2(\varsigma_k,x(s_k^-)) \rVert^2_{\beta} \nonumber \\
& \leq M^2 T^2 D_{h_k^2}\mathbb{E} \lVert y(s_k^-)-x(s_k^-) \rVert_{\beta}^2.
\end{align}
\begin{align*}
I_3 &=  \mathbb{E}\lVert \int_{\varsigma_k}^s  P_{\alpha}(s-\varsigma) [\mathcal{K}_n(\varsigma,y(\varsigma)) - \mathcal{K}_n(\varsigma,x(\varsigma))] \rVert_{\beta}^2.
\end{align*}
By the assumption $\textbf{(A2)}$ and the Hölder's inequality, we get
\begin{align}
I_3 &\leq \lVert \mathcal{A}^{\beta-1} \rVert^2 \int_{\varsigma_k}^s \lVert \mathcal{A}P_{\alpha}(s-\varsigma) \rVert^2 \mathbb{E} \lVert \mathcal{K}_n(\varsigma,y(\varsigma)) - \mathcal{K}_n(\varsigma,x(\varsigma)) \rVert^2 d{\varsigma} \nonumber \\
& \leq \lVert \mathcal{A}^{\beta-1} \rVert^2 \rho^2 (s-\varsigma_k){L_\mathcal{K}} \mathbb{E} \lVert y(\varsigma)-x(\varsigma) \rVert^2_{\beta}.
\end{align}
\begin{align}
I_4 = \mathbb{E}\lVert \int_{\varsigma_k}^s P_{\alpha}(s-\varsigma) \int_0^\varsigma a(\varsigma-r)[\mathcal{N}_n(r,y(r))-\mathcal{N}_n(r,x(r))]dW(r) d{\varsigma} \rVert^2_{\beta}. \nonumber
\end{align}
By the Lemma (\ref{lemma}), the assumption $\textbf{(A3)}$ and Hölder's inequality, we get
\begin{align}\label{3.10}
I_4 &\leq \lVert \mathcal{A}^{\beta-1} \rVert^2 \int_{\varsigma_k}^s \lVert \mathcal{A}P_{\alpha}(s-\varsigma) \rVert^2 \mathbb{E}\ \bigg\{\lVert \int_0^\varsigma a(\varsigma-r)[\mathcal{N}_n(r,y(r))-\mathcal{N}_n(r,x(r))] \nonumber \\
&\quad \times dW(r) \rVert^2\bigg\}d{\varsigma} \nonumber \\
& \leq 4Tr(Q) \lVert \mathcal{A}^{\beta-1} \rVert^2 \rho^2 (s-\varsigma_k)\int_0^T |a(\varsigma-r)|^2 \mathbb{E}\lVert \mathcal{N}_n(r,y(r))-\mathcal{N}_n(r,x(r))] \rVert^2_{L_2^0} dr\nonumber \ \\
&\leq 4Tr(Q) \lVert \mathcal{A}^{\beta-1} \rVert^2 \rho^2 (s-{\varsigma}_k) \lVert a^2 \rVert_{L^p(0,T)} \lVert L_\mathcal{N} \rVert_{L^q(0,T)}\mathbb{E} \lVert y(r)-x(r) \rVert^2_{\beta}.
\end{align}
Thus by (\ref{3.7})-(\ref{3.10}), we get
\begin{equation}\label{3.11}
\mathbb{E} \lVert (\phi_n y)(s) - (\phi_n x)(s) \rVert^2_{\beta} \leq  Q_k \lVert y-x \rVert^2_{PC,\beta}.
\end{equation}
Therefore, by the inequalities (\ref{3.4}),(\ref{3.5}) and (\ref{3.11}), we get
\begin{equation}
\lVert (\phi_n y) - (\phi_n x) \rVert^2_{PC,\beta} \leq D \lVert y-x \rVert^2_{PC,\beta}.
\end{equation}
Hence, $\phi_n$ is a contraction map on $B_R(PC_T^{\beta})$. Therefore, by Banach contraction principal, there exists a unique $y_n \in B_R(PC_T^{\beta})$ satisfying
\begin{equation}\label{3.13}y_n(s)= 
\begin{cases}
C_{\alpha}(s)y_0 + S_{\alpha}(s)z_0 + \int_0^s P_{\alpha}(s-\varsigma)[\mathcal{K}_n(\varsigma,y_n(\varsigma))\\
\qquad\qquad+\int_0^{\varsigma} a(\varsigma-r)\mathcal{N}_n(r,y_n(r)) dW(r)]d{\varsigma},\ s \in (0,s_1];\\
 h_{k,n}^1(s,y_n(s_k^-)), \ s \in (s_k,\varsigma_k],  k = 1,2,\cdots,q;\\
C_{\alpha}(s-\varsigma_k)h_{k,n}^1(\varsigma_k,y_n(s_k^-))+S_{\alpha}(s-\varsigma_k)h_{k,n}^2(\varsigma_k,y_n(s_k^-))\\
\quad + \int_{\varsigma_k}^s P_{\alpha}(s-\varsigma)[\mathcal{K}_n(\varsigma,y_n(\varsigma))+ \int_0^{\varsigma} a(\varsigma-r)\mathcal{N}_n(r,y_n(r)) dW(r)]d{\varsigma}, \\ \qquad\qquad\qquad\qquad\qquad s \in (\varsigma_k,s_{k+1}],\ k = 1,2,\cdots,q.
\end{cases}
\end{equation}

\end{proof}
\end{theorem}
\begin{Corollary} Suppose that $\textbf{(A1)}-\textbf{(A4)}$ are fulfilled, $z_0,y_0 \in D(\mathcal{A})$ and $h_{k,n}^i(s,y_{n}(s_k^-)) \in D(\mathcal{A}) \ \forall$ $s \in (s_k,\varsigma_k], \ k = 1,2,\cdots,q, \ i = 1,2$ then $y_n(s) \in D(\mathcal{A}^{\eta}) \ \forall s \in [0,T]$, where $0 \leq \eta \leq 1.$
\end{Corollary}
\begin{proof}
This result is true for $s\in(s_k,\varsigma_k]$. Now, if $y_0,z_0 \in D(\mathcal{A})$ then $C_{\alpha}(s)y_0, \ S_{\alpha}(s)z_0 \in D(\mathcal{A})$ therefore their sum is also in $D(\mathcal{A})$. Similarly, if $h_{k,n}^i(s,y_n(s_k^-)) \in D(\mathcal{A})$ for all \ $s \in (s_k,\varsigma_k], k = 1,2,\cdots,q$ and i = 1,2 then  $C_{\alpha}(s-\varsigma_k)h_{k,n}^1(\varsigma_k,y_n(s_k^-)), \ S_{\alpha}(s-\varsigma_k)h_{k,n}^2(\varsigma_k,y_n(s_k^-)) \in D(\mathcal{A})$ therefore their sum is also in $D(\mathcal{A})$. From proposition (3.3) in \cite{5}, $\int_0^s P_{\alpha}(s-\varsigma)\mathcal{K}_n(\varsigma,y_n(\varsigma)) d{\varsigma} \in D(\mathcal{A})$ for all $\mathcal{K}_n(\varsigma,y_n(\varsigma)) \in \mathcal{H}$. Hence, the the requisite result follows from the aforementioned facts as well as the result that $D(\mathcal{A}) \subseteq D(\mathcal{A}^{\eta})$ for $0 \leq \eta \leq 1$.
\end{proof} 
\begin{lemma}
Let $\textbf{(A1)}-\textbf{(A4)}$ hold and $y_0,z_0 \in D(\mathcal{A})$ then there exists a constant $M^{\prime}$, independent of n, such that
\begin{equation*}
\mathbb{E} \lVert \mathcal{A}^{\eta}y_n(s) \rVert^2 \ \leq \ M^{\prime}, \quad \forall \ s \in [0,T] \quad  \text{and} \  0 < \eta < 1 
\end{equation*} 
\text{where}
\begin{equation*}
M^{\prime} = \max\limits_{1 \leq k \leq q}\big\{\max\limits_{0 \leq k \leq q} M_k, C_{h_k^1} \big\}.
\end{equation*}
\end{lemma}
\begin{proof}
For any $s \in (0,s_1]$, applying $\mathcal{A}^{\eta}$ on both sides of (\ref{3.13}) and taking norm
\begin{align}
\mathbb{E}\lVert y_n(s) \rVert_{\eta}^2 & \leq 4\bigg[\lVert C_{\alpha}(s) \rVert^2 \lVert y_0 \rVert^2_{\eta} + \lVert S_{\alpha}(s)\rVert^2 \lVert z_0 \rVert^2_{\eta}+ \lVert \mathcal{A}^{\eta-1} \rVert^2 \int_0^s \lVert \mathcal{A}P_{\alpha}(s-\varsigma) \rVert^2  \nonumber\\
& \quad \times\mathbb{E}\lVert \mathcal{K}_n(\varsigma,y_n(\varsigma)) \rVert^2 d{\varsigma}+ 4Tr(Q)\lVert \mathcal{A}^{\eta-1} \rVert^2 \int_0^s \lVert \mathcal{A}P_{\alpha}(s-\varsigma) \rVert^2 \nonumber \\
& \quad \times \bigg(\int_0^T |a(\varsigma-r)|^2 \mathbb{E} \lVert \mathcal{N}_n(r,y_n(r)) \rVert^2_{L_2^0} dr \bigg) d\varsigma \bigg] \nonumber\\
& \leq 4\bigg[M^2 \lVert y_0 \rVert^2_{\eta}+M^2 T^2 \lVert z_0 \rVert^2_{\eta}+\lVert \mathcal{A}^{\eta-1} \rVert^2 \rho^2 s_1 \nonumber \\
& \quad \times \bigg\{{L^{\prime}_\mathcal{K}}+4 Tr(Q) \lVert a^2 \rVert_{L^p(0,T)} \lVert {L^{\prime}_\mathcal{N}} \rVert_{L^q(0,T)}\bigg\}\bigg]\nonumber\\
&= M_0 \leq \ M^{\prime}.
\end{align}
Now, for any $ s \in (s_k,\varsigma_k], \ k = 1,2,\cdots,q,$ after applying $\mathcal{A}^{\eta}$ on both sides of (\ref{3.13}), we have
\begin{align}
\mathbb{E}\lVert y_n(s) \rVert^2_{\eta} = \mathbb{E}\lVert h^1_{k,n}(s,y_n(s_k^-)) \rVert^2_{\eta} \ \leq \ C_{h_k^1}.
\end{align}
Similarly, for any $s \in (\varsigma_k,s_{k+1}], \ k = 1,2,\cdots,q$, after applying $\mathcal{A}^{\eta}$ on both sides of (\ref{3.13}), we have
\begin{align}
\mathbb{E}\lVert y_n(s) \rVert^2_{\eta} & \leq \ 4\bigg[M^2 C_{h_k^1} + M^2T^2C_{h_k^2}+  \lVert \mathcal{A}^{\eta-1} \rVert^2 \int_{\varsigma_k}^s \lVert \mathcal{A}P_{\alpha}(s-\varsigma) \rVert^2 \mathbb{E}\lVert \mathcal{K}_n(\varsigma,y_n(\varsigma)) \rVert^2 d{\varsigma} \nonumber \\
& \quad + 4Tr(Q)\lVert \mathcal{A}^{\eta-1} \rVert^2 \int_{\varsigma_k}^s \lVert \mathcal{A}P_{\alpha}(s-\varsigma) \rVert^2 \bigg(\int_0^T |a(\varsigma-r)|^2 \mathbb{E} \lVert \mathcal{N}_n(r,y_n(r)) \rVert^2_{L_2^0} \nonumber\\
& \quad \times  dr \bigg) d{\varsigma}\bigg] \nonumber \\
&\leq 4 \bigg[M^2  C_{h_k^1} +M^2 T^2  C_{h_k^2} +\lVert \mathcal{A}^{\eta-1} \rVert^2 \rho^2 s_{k+1} \nonumber \\
& \quad \times\bigg\{{L^{\prime}_{\mathcal{K}}} +4 Tr(Q) \lVert a^2 \rVert_{L^p(0,T)} \lVert {L^{\prime}_\mathcal{N}} \rVert_{L^q(0,T)}\bigg\}\bigg] \nonumber \\
&= M_k \leq \ M^{\prime}.
\end{align}
\end{proof}
\section{Convergence of approximate solutions}\label{section 4}
The convergence of the approximate solutions $y_n \in B_R(PC_T^{\beta})$ to a unique mild solution $y$, is shown in this section.
\begin{theorem}
\label{thrm4.1}
Let $\textbf{(A1)-(A4)}$ satisfied. If $y_0, z_0 \in D(\mathcal{A})$ then\\
$\lim\limits_{m \rightarrow \infty} \sup\limits_{\{n \geq m, \ 0 \leq s \leq T\}} \mathbb{E}\lVert y_n(s) - y_m(s) \rVert^2_{\beta} = 0$.\\
Therefore, $\{y_n\}$ form a Cauchy sequence in $B_R(PC_T^{\beta})$ which converges to $y$.
\end{theorem}
\begin{proof}
Let $n \geq m$ and $0 < \beta < \eta < 1$. Since $\mathcal{H}_n \supset \mathcal{H}_m$ and $\mathcal{H}_n^{\perp} \subset \mathcal{H}_m^{\perp}$, where  $\mathcal{H}_m^{\perp}$  and $\mathcal{H}_n^{\perp}$ be the orthogonal complement of $\mathcal{H}_m$ and $\mathcal{H}_n$ respectively, for all $n,m \in \mathbb{N} \cup \{0\}$. Then $\mathcal{H}$ can be written as $\mathcal{H} = \mathcal{H}_m \oplus \mathcal{H}_m^{\perp} = \mathcal{H}_n \oplus \mathcal{H}_n^{\perp}$. Any arbitrary $z \in \mathcal{H}$ has a unique representation $z = z_m + x_m$, where $z_m \in \mathcal{H}_m$ and $x_m \in \mathcal{H}_m^{\perp}$. Then $\mathcal{P}^{m}z = z_m \in \mathcal{H}_m$. Since $x_m \in \mathcal{H}_m^{\perp}$ therefore $x_m = \sum_{j = m+1}^n a_j \psi_{j} + x_m^{\prime},$ where $x_m^{\prime} \in \mathcal{H}_n^{\perp}$.\\
 Let $ \sum_{j = m+1}^n a_j \psi_j = z_m^{\prime}$.\\
Therefore
\begin{equation*}
\mathcal{P}^{n}z - \mathcal{P}^{m}z = z_m^{\prime} = \sum_{j = m+1}^n a_j \psi_j.
\end{equation*}
If $z = \sum_{j = 1}^{\infty} a_j \psi_j$ then $\lVert z \rVert^2 = \sum_{j = 1}^{\infty} |a_j|^2$.\\
Since $\mathcal{A}^{\beta - \eta} \psi_j = \lambda_j^{\beta-\eta} \psi_j$, thus we arrive at the following conclusion after some calculations
\begin{align}
\mathbb{E}\lVert \mathcal{A}^{\beta-\eta}(\mathcal{P}^n - \mathcal{P}^m)z \rVert^2
& \leq \frac{1}{\lambda_{m}^{2(\eta-\beta)}} \lVert z \rVert^2. \nonumber
\end{align}
Therefore
\begin{align}\label{4.1}
\mathbb{E} \lVert \mathcal{A}^{\beta - \eta}(\mathcal{P}^n - \mathcal{P}^m) \mathcal{A}^{\eta}y_m(\varsigma) \rVert^2 &\leq \frac{1}{\lambda_{m}^{2(\eta-\beta)}} \mathbb{E} \lVert \mathcal{A}^{\eta}y_m(\varsigma) \rVert^2 \nonumber \\
& \leq \frac{1}{\lambda_m^{2(\eta-\beta)}}M^{\prime}. 
\end{align}
Using $(\textbf{A2})$ and (\ref{4.1}), we get
\begin{align}\label{4.2}
&\mathbb{E}\lVert \mathcal{K}_n(\varsigma,y_n(\varsigma))-\mathcal{K}_m(\varsigma,y_m(\varsigma)) \rVert^2\nonumber\\
 &\leq 2\bigg[\mathbb{E} \lVert \mathcal{K}_n(\varsigma,y_n(\varsigma))-\mathcal{K}_n(\varsigma,y_m(\varsigma)) \rVert^2+ \mathbb{E}\lVert \mathcal{K}_n(\varsigma,y_m(\varsigma))-\mathcal{K}_m(\varsigma,y_m(\varsigma)) \rVert^2 \bigg]\nonumber\\
& \leq 2L_\mathcal{K}\bigg[\mathbb{E}\lVert y_n(\varsigma)-y_m(\varsigma) \rVert^2_{\beta}+\mathbb{E}\lVert (\mathcal{P}^n-\mathcal{P}^m)y_m(\varsigma) \rVert^2_{\beta}\bigg]\nonumber \\
& \leq 2L_\mathcal{K}\bigg[\lVert y_n-y_m \rVert^2_{PC,\beta}+\frac{M^{\prime}}{\lambda_m^{2(\eta-\beta)}} \bigg].
\end{align}
Similarly, $(\textbf{A3})$ and (\ref{4.1}) imply that
\begin{align}\label{4.3}
&\mathbb{E}\lVert \mathcal{N}_n(\varsigma,y_n(\varsigma))-\mathcal{N}_m(\varsigma,y_m(\varsigma)) \rVert^2_{L_2^0} \nonumber\\
&\leq 2 \bigg[\mathbb{E}\lVert \mathcal{N}_n(\varsigma,y_n(\varsigma)) - \mathcal{N}_n(\varsigma,y_m(\varsigma)) \rVert^2_{L_2^0} \quad +\mathbb{E}\lVert \mathcal{N}_n(\varsigma,y_m(\varsigma))-\mathcal{N}_m(\varsigma,y_m(\varsigma)) \rVert^2_{L_2^0} \bigg] \nonumber\\
&\leq 2L_\mathcal{N}(s) \bigg[\mathbb{E}\lVert y_n(\varsigma)-y_m(\varsigma) \rVert^2_{\beta}+\mathbb{E}\lVert (\mathcal{P}^n-\mathcal{P}^m)y_m(\varsigma) \rVert^2_{\beta}\bigg]\nonumber \\
&\leq  2L_\mathcal{N}(s) \bigg[\lVert y_n-y_m \rVert^2_{PC,\beta}+\frac{M^{\prime}}{\lambda_m^{2(\eta-\beta)}} \bigg].
\end{align}
Also, $(\textbf{A4})$ and (\ref{4.1}) yields that
\begin{align}\label{4.4}
&\mathbb{E}\lVert h_{k,n}^i(s,y_n(s_k^-))- h_{k,m}^i(s,y_m(s_k^-)) \rVert^2_{\beta} \nonumber\\
&\leq 2\bigg[\mathbb{E} \lVert h_{k,n}^i(s,y_n(s_k^-))-h_{k,n}^i(s,y_m(s_k^-)) \rVert^2_{\beta} + \mathbb{E} \lVert h_{k,n}^i(s,y_m(s_k^-))-h_{k,m}^i(s,y_m(s_k^-)) \rVert^2_{\beta} \bigg] \nonumber \\
&\leq 2D_{h_k^i} \bigg[ \mathbb{E}\lVert y_n(s_k^-)) - y_m(s_k^-)) \rVert^2_{\beta}+\mathbb{E} \lVert (\mathcal{P}^n-\mathcal{P}^m)y_m(s_k^-) \rVert^2_{\beta} \bigg] \nonumber\\
& \leq 2D_{h_k^i} \bigg[\lVert y_n-y_m \rVert^2_{PC,\beta}+\frac{M^{\prime}}{\lambda_m^{2(\eta-\beta)}} \bigg].
\end{align}
Now, for any $s \in (0,s_1]$, by using (\ref{4.2}) and (\ref{4.3}) we compute
\begin{align}
&\mathbb{E} \lVert y_n(s)-y_m(s) \rVert^2_{\beta} \\
& \qquad\quad\leq 2 \bigg\{ \int_0^s \lVert \mathcal{A}^{\beta-1} \rVert^2 \lVert \mathcal{A}P_{\alpha}(s-\varsigma) \rVert^2 \bigg[\mathbb{E}\lVert \mathcal{K}_n(\varsigma,y_n(\varsigma)) - \mathcal{K}_m(\varsigma,y_m(\varsigma)) \rVert^2 + 4Tr(Q) \nonumber\\
&\qquad\qquad\times \int_0^T |a(\varsigma-r)|^2 \mathbb{E}\lVert \mathcal{N}_n(r,y_n(r))-\mathcal{N}_m(r,y_m(r)) \rVert^2_{L_2^0} dr \bigg]d{\varsigma} \bigg\}\nonumber\\
&\qquad\quad\leq 4\lVert \mathcal{A}^{\beta-1} \rVert^2 \rho^2s_1(L_\mathcal{K}+4TrQ \lVert a^2 \rVert_{L^p{(0,T)}}\lVert L_\mathcal{N} \rVert_{L^q(0,T)}) \nonumber \\
& \qquad\qquad \times \big(\lVert y_n-y_m \rVert^2_{PC,\beta}+\frac{M^{\prime}}{\lambda_m^{2(\eta-\beta)}}\big)\nonumber.
\end{align}
Hence
\begin{align}\label{4.5}
\lVert y_n-y_m \rVert^2_{PC,\beta} & \leq \frac{2M^{\prime}Q_0}{(1-2Q_0)}\frac{1}{\lambda_m^{2(\eta-\beta)}}.
\end{align}
Similarly, for any $s \in (s_k,\varsigma_k], k = 1,2,\cdots,q,$ we use (\ref{4.4}) and compute
\begin{align}
\mathbb{E}\lVert y_n(s)-y_m(s) \rVert^2_{\beta} &= \mathbb{E}\lVert h_{k,n}^1(s,y_n(s_k^-))-h_{k,m}^1(s,y_m(s_k^-)) \rVert^2_{\beta} \nonumber\\
&\leq 2 D_{h_k^1} \bigg[\lVert y_n-y_m \rVert^2_{PC,\beta}+\frac{M^{\prime}}{\lambda_m^{2(\eta-\beta)}} \bigg].\nonumber
\end{align}
Hence
\begin{align}\label{4.6}
\lVert y_n-y_m \rVert^2_{PC,\beta} \leq \frac{2D_{h_k^1}M^{\prime}}{(1-2D_{h_k^1})}\frac{1}{\lambda_m^{2(\eta-\beta)}}.
\end{align} 
Also, for any $s \in (\varsigma_k,s_{k+1}], k = 1,2,\cdots,q$, we use (\ref{4.2})-(\ref{4.4}) and compute
\begin{align}\label{4.7}
&\mathbb{E}\lVert y_n(s)-y_m(s) \rVert^2_{\beta} \nonumber \\
& \leq 4\mathbb{E}\bigg\{\lVert C_{\alpha}(s-\varsigma_k) [h_{k,n}^1(\varsigma_k,y_n(s_k^-))-h_{k,m}^1(\varsigma_k,y_m(s_k^-))] \rVert^2_{\beta} \nonumber \\
&\quad+\lVert S_{\alpha}(s-\varsigma_k)[h_{k,n}^2(\varsigma_k,y_n(s_k^-)) -h_{k,m}^2(\varsigma_k,y_m(s_k^-))] \rVert^2_{\beta}\nonumber\\
& \quad + \int_{\varsigma_k}^s \lVert \mathcal{A}^{\beta}P_{\alpha}(s-\varsigma)[\mathcal{K}_n(\varsigma,y_n(\varsigma))-\mathcal{K}_m(\varsigma,y_m(\varsigma))] \rVert^2 d{\varsigma}\nonumber \\
&\quad +  \int_{\varsigma_k}^s \lVert \mathcal{A}^{\beta}P_{\alpha}(s-\varsigma)\int_0^\varsigma a(\varsigma-r)[\mathcal{N}_n(r,y_n(r))-\mathcal{N}_m(r,y_m(r))]dW(r) \rVert^2 d{\varsigma}\bigg\}. \nonumber\\
& = 4(\xi_1+\xi_2+\xi_3+\xi_4).
\end{align}
We calculate the values of $\xi_1, \xi_2, \xi_3$ and $\xi_4$, as follows,
\begin{align}\label{4.8}
\xi_1 & \leq  \lVert C_{\alpha}(s-\varsigma_k) \rVert^2 \mathbb{E}\lVert h_{k,n}^1(\varsigma_k,y_n(s_k^-)) - h_{k,m}^1(\varsigma_k,y_m(s_k^-)) \rVert^2_{\beta} \nonumber \\
& \leq 2M^2 D_{h_k^1}[\lVert y_n-y_m \rVert^2_{PC,\beta}+\frac{M^{\prime}}{\lambda_m^{2(\eta-\beta)}}].
\end{align}
\begin{align}\label{4.9}
\xi_2 & \leq \lVert S_{\alpha}(s-\varsigma_k) \rVert^2 \mathbb{E}\lVert h_{k,n}^2(\varsigma_k,y_n(s_k^-)) - h_{k,m}^2(\varsigma_k,y_m(s_k^-)) \rVert^2_{\beta} \nonumber \\
& \leq 2M^2 T^2 D_{h_k^2}[\lVert y_n-y_m \rVert^2_{PC,\beta}+\frac{M^{\prime}}{\lambda_m^{2(\eta-\beta)}}]
\end{align}
\begin{align}\label{4.10}
\xi_3 &\leq  \int_{\varsigma_k}^s \lVert \mathcal{A}^{\beta-1} \rVert^2  \lVert \mathcal{A}P_{\alpha}(s-\varsigma) \rVert^2 \mathbb{E} \lVert \mathcal{K}_n(\varsigma,y_n(\varsigma)) - \mathcal{K}_m(\varsigma,y_m(\varsigma)) \rVert^2 d{\varsigma} \nonumber \\
& \leq 2\lVert \mathcal{A}^{\beta-1} \rVert^2 \rho^2 s_{k+1}L_\mathcal{K} [\lVert y_n-y_m \rVert^2_{PC,\beta}+\frac{M^{\prime}}{\lambda_m^{2(\eta-\beta)}}]
\end{align}
\begin{align}\label{4.11}
\xi_4 &\leq \int_{\varsigma_k}^s \lVert \mathcal{A}^{\beta-1} \rVert^2  \lVert \mathcal{A}P_{\alpha}(s-\varsigma) \rVert^2 \mathbb{E} \lVert \int_0^s a(\varsigma-r) [\mathcal{N}_n(\varsigma,y_n(\varsigma))-\mathcal{N}_m(\varsigma,y_m(\varsigma))] \nonumber\\
&\quad \times dW(r) \rVert^2 d\varsigma \nonumber \\
& \leq 4Tr(Q) \lVert \mathcal{A}^{\beta-1} \rVert^2 \rho^2 s_{k+1}\int_0^T |a(\varsigma-r)|^2 \mathbb{E}\lVert \mathcal{N}_n(r,y_n(r))-\mathcal{N}_m(r,y_m(r)) \rVert^2_{L_2^0} dr\nonumber\\
& \leq 8 Tr(Q)\lVert \mathcal{A}^{\beta-1} \rVert^2 \rho^2 s_{k+1}\lVert a^2 \rVert_{L^p(0,T)} \lVert L_\mathcal{N} \rVert_{L^q(0,T)}[\lVert y_n-y_m \rVert^2_{PC,\beta}+\frac{M^{\prime}}{\lambda_m^{2(\eta-\beta)}}].
\end{align}
From (\ref{4.8})-(\ref{4.11}), we get the values of $\xi_1, \xi_2, \xi_3$ and $\xi_4$. On substituting these values in (\ref{4.7}), we obtain
\begin{align}\label{4.12}
\lVert y_n-y_m \rVert^2_{PC,\beta} \leq \frac{2M^{\prime}Q_k}{(1-2Q_k)} \frac{1}{\lambda_m^{2(\eta-\beta)}}.
\end{align}
Therefore, from inequalities (\ref{4.5}), (\ref{4.6}) and (\ref{4.12})
\begin{equation*}
\lim_{m \rightarrow \infty} \sup\limits_{\{n \geq m, 0 \leq s \leq T\}}\mathbb{E}\lVert y_n(s)-y_m(s) \rVert^2_{\beta} = 0.
\end{equation*}
Since $\frac{1}{\lambda_m^{2(\eta-\beta)}} \rightarrow 0$ as $m \rightarrow \infty$. Hence, we get the desired result. 
\end{proof}
Now, we can establish the following convergence theorem with the aid of Theorem \ref{thrm3.1} and Theorem~\ref{thrm4.1}.
\begin{theorem}\label{thrm4.2}
Let all the assumption $\textbf{(A1)-(A4)}$ hold and $y_0, z_0 \in D(\mathcal{A})$. Then there exists a unique $y_n \in PC_T^{\beta}$ and $y$ satisfying (\ref{2.3}) such that $y_n \rightarrow y$ in $PC_T^{\beta}$ as $n \rightarrow \infty$.
\end{theorem}
\begin{proof}
One can see existence and convergence of $y_n$ on [0,T] in Theorem \ref{thrm3.1} and Theorem \ref{thrm4.1}. Now, it only remains to show that
the limit $y$ of $y_n$ is given by system (\ref{2.3}). We have the following inequalities
\begin{align}\label{4.13}
&\mathbb{E}\lVert \mathcal{K}_n(\varsigma,y_n(\varsigma))-\mathcal{K}(\varsigma,y(\varsigma)) \rVert^2 \nonumber \\
&\leq 2 \bigg[\mathbb{E}\lVert \mathcal{K}_n(\varsigma,y_n(\varsigma))-\mathcal{K}_n(\varsigma,y(\varsigma)) \rVert^2+ \mathbb{E}\lVert \mathcal{K}_n(\varsigma,y(\varsigma))-\mathcal{K}(\varsigma,y(\varsigma)) \rVert^2 \bigg] \nonumber\\
&\leq 2L_\mathcal{K}\bigg[\mathbb{E}\lVert y_n(\varsigma)-y(\varsigma) \rVert^2_{\beta}+\mathbb{E} \lVert (\mathcal{P}^n-I)y(\varsigma) \rVert^2_{\beta}\bigg]\nonumber\\
&\leq 2L_\mathcal{K}\bigg[ \lVert y_n-y \rVert^2_{PC,\beta}+\lVert (\mathcal{P}^n-I)y \rVert^2_{PC,\beta}\bigg].
\end{align}
Also, using the assumption $(\textbf{A3})$
\begin{align}\label{4.14}
&\mathbb{E}\lVert \mathcal{N}_n(s,y_n(s))-\mathcal{N}(s,y(s)) \rVert^2_{L_2^0} \nonumber\\
&\leq 2 \bigg[\mathbb{E}\lVert \mathcal{N}_n(s,y_n(s))-\mathcal{N}_n(s,y(s)) \rVert^2_{L_2^0}+ \mathbb{E}\lVert \mathcal{N}_n(s,y(s))-\mathcal{N}(s,y(s)) \rVert^2_{L_2^0} \bigg] \nonumber\\
&\leq 2L_\mathcal{N}(s)\bigg[\mathbb{E}\lVert y_n(s)-y(s) \rVert^2_{\beta}+\mathbb{E} \lVert (\mathcal{P}^n-I)y(s) \rVert^2_{\beta}\bigg]\nonumber\\
&\leq 2L_\mathcal{N}(s)\bigg[ \lVert y_n-y \rVert^2_{PC,\beta}+\lVert (\mathcal{P}^n-I)y \rVert^2_{PC,\beta}\bigg].
\end{align}
Similarly, for any $s \in (s_k,\varsigma_k], k = 1,2,\cdots,q, \ i =1,2$, and the assumption $(\textbf{A4})$ implies that
\begin{align}\label{4.15}
\mathbb{E}\lVert h_{k,n}^i(s,y_n(s_k^-))-h_{k}^i(s,y(s_k^-)) \rVert^2_{\beta} &\leq 2D_{h_k^i}\bigg[\lVert y_n-y \rVert^2_{PC,\beta}+\lVert (P^n-I)y \rVert^2_{PC,\beta}\bigg].
\end{align}
Since $y_n \rightarrow y$ and $\mathcal{P}^n y \rightarrow y$ as $n \rightarrow \infty$ therefore the right hand side of (\ref{4.13}) - (\ref{4.15}) converges to zero as $n \rightarrow \infty$. 
Therefore, for any $s \in (0,s_1]$
\begin{align*}
\mathbb{E}\lVert y_n(s)-y(s) \rVert^2_{\beta} &\leq 2\bigg\{\int_0^s \lVert \mathcal{A}^{\beta-1} \rVert^2 \lVert \mathcal{A}P_{\alpha}(s-\varsigma) \rVert^2 \bigg[\mathbb{E}\lVert \mathcal{K}_n(\varsigma,y_n(\varsigma))-\mathcal{K}(\varsigma,y(\varsigma)) \rVert^2 \nonumber \\
&\quad + 4Tr(Q) \int_0^T |a(\varsigma-r)|^2 \mathbb{E}\lVert \mathcal{N}_n(r,y_n(r)) \nonumber \\
&\quad-\mathcal{N}(r,y(r)) \rVert^2_{L_2^0} dr \bigg]d{\varsigma} \bigg\} \rightarrow 0 \ \text{as} \ n \rightarrow \infty.
\end{align*}
Also, for any $s \in (\varsigma_k,s_{k+1}], k = 1,2,\cdots,q$
\begin{align}
\mathbb{E}\lVert y_n(s)-y(s) \rVert^2_{\beta} &\leq 4\bigg[\lVert C_{\alpha}(s-\varsigma_k) \rVert^2 \mathbb{E}\lVert h_{k,n}^1(\varsigma_k,y_n(s_k^-))-h_k(\varsigma_k,y(s_k^-)) \rVert^2_{\beta}\nonumber\\
&\quad  + \lVert S_{\alpha}(s-\varsigma_k) \rVert^2 \mathbb{E}\lVert h_{k,n}^2(\varsigma_k,y_n(s_k^-))-h_{k}^2(\varsigma_k,y(s_k^-)) \rVert^2_{\beta} \nonumber\\
& \quad+\int_{\varsigma_k}^s \lVert \mathcal{A}^{\beta}P_{\alpha}(s-\varsigma) \rVert^2 \mathbb{E}\lVert \mathcal{K}_n(\varsigma,y_n(\varsigma))-\mathcal{K}(\varsigma,y(\varsigma)) \rVert^2 d\varsigma \nonumber\\
&\quad + \int_{\varsigma_k}^\varsigma \lVert \mathcal{A}^{\beta}P_{\alpha}(s-\varsigma) \rVert^2 (4TrQ \int_0^T |a(\varsigma-r)|^2 \nonumber\\
& \quad \times \mathbb{E}\lVert \mathcal{N}_n(r,y_n(r))-\mathcal{N}(r,y(r)) \rVert^2_{L_2^0} dr)d\varsigma \ \rightarrow 0 \ \text{as} \ n \rightarrow \infty.\nonumber 
\end{align}
As a result of the preceding inequalities, we may deduce that $y$ is given by system (\ref{2.3}). The proof
is now complete.
\end{proof}
\section{Faedo-Galerkin approximations}
If we project (\ref{3.13}) onto $\mathcal{H}_n$, we get a finite dimensional approximation i.e. $\bar{y}_n = \mathcal{P}^ny_n$, which is
known as F-G approximate solution satisfying
\begin{equation}\bar{y}_n(s)=
\begin{cases}
C_{\alpha}(s)\mathcal{P}^ny_0 + S_{\alpha}(s)\mathcal{P}^nz_0 + \int_0^s P_{\alpha}(s-\varsigma)[\mathcal{P}^n\mathcal{K}_n(\varsigma,y_n(\varsigma))\\
\qquad\qquad+\int_0^\varsigma a(\varsigma-r)\mathcal{P}^n\mathcal{N}_n(r,y_n(r)) dW(r)]d\varsigma,\ s \in (0,s_1];\\
\mathcal{P}^nh_{k,n}^1(s,y_n(s_k^-)), \ s \in (s_k,\varsigma_k], \ k = 1,2,\cdots,q;\\
 C_{\alpha}(s-\varsigma_k)\mathcal{P}^nh_{k,n}^1(\varsigma_k,y_n(s_k^-))+S_{\alpha}(s-\varsigma_k)\mathcal{P}^nh_{k,n}^2(\varsigma_k,y_n(s_k^-))\\ 
\quad +\int_{\varsigma_k}^s P_{\alpha}(s-\varsigma)[\mathcal{P}^n \mathcal{K}_n(\varsigma,y_n(\varsigma))+\int_0^{\varsigma} a(\varsigma-r)\mathcal{P}^n \mathcal{N}_n(r,y_n(r)) \\
\qquad\qquad\qquad \times  dW(r)]d\varsigma, \ s \in (\varsigma_k,s_{k+1}], \ k = 1,2,\cdots,q.
\end{cases}
\end{equation}
Solutions $y$ and $\bar{y}_n$ given by (\ref{2.3}) and (\ref{3.13}) respectively, have the following representations
\begin{align*}
 y(s) &= \sum_{j=0}^{\infty} \eta_j(s) \psi_j, \quad  \eta_j(s) = \langle  y(s), \psi_j \rangle , \  j= 0,1,2,\cdots;\\
\bar{y}_n(s) &= \sum_{j=0}^{n} \eta_j^n(s) \psi_j, \quad \eta_j^n(s) = \langle \bar{y}_n(s), \psi_j \rangle, \  j = 0,1,2,\cdots,n.
\end{align*}
\begin{theorem}\label{thrm5.1}
Let the assumptions $\textbf{(A1)-(A4)}$ satisfied and $y_0, z_0 \in D(\mathcal{A})$. Then,
\begin{eqnarray*}
\lim_{m \rightarrow \infty} \sup_{\{n \geq m, \ 0 \leq s \leq T\}} \mathbb{E} \lVert \bar{y}_n(s) - \bar{y}_m(s) \rVert_{\beta}^2 = 0.
\end{eqnarray*}
\begin{proof}
For $n \geq m$, we have
\begin{align*}
\mathbb{E}\lVert \bar{y}_n(s) -\bar{y}_m(s) \rVert_{\beta}^2 &= \mathbb{E} \lVert \mathcal{P}^n y_n(s) - \mathcal{P}^m y_m(s) \rVert_{\beta}^2 \\
& \leq 2\mathbb{E}\bigg[\lVert \mathcal{P}^n (y_n(s) - y_m(s)) \rVert_{\beta}^2 + \lVert (\mathcal{P}^n - \mathcal{P}^m ) y_m(s) \rVert_{\beta}^2 \bigg] \\
& \leq 2\bigg[\mathbb{E} \lVert y_n(s) - y_m(s) \rVert_{\beta}^2 + \frac{M^{\prime}}{\lambda_m^{2(\eta-\beta)}} \bigg].
\end{align*}
Now take the supremum over all $0 \leq s \leq T$, we get 
\begin{eqnarray*}
\sup_{\{0 \leq s \leq T\}}\mathbb{E} \lVert \bar{y}_n(s) - \bar{y}_m(s) \rVert_{\beta}^2 \leq 2\bigg[ \lVert y_n-y_m \rVert^2_{PC,\beta} + \frac{M^{\prime}}{\lambda_m^{2(\eta-\beta)}}\bigg].
\end{eqnarray*}
Thus, we achieve the desired outcome by using the Theorem (\ref{thrm4.1}).
\end{proof}
\end{theorem}
Next, the existence of F-G approximate solution and their convergence can be demonstrated by the following theorem. 
\begin{theorem}\label{thrm5.2}
Let $\textbf{(A1)-(A4)}$ satisfied. Then there exist functions $\bar{y}_n \in PC^{\beta}_T$ and $y \in PC^{\beta}_T$ such that $\bar{y}_n \rightarrow y$ as $n \rightarrow \infty$ in $PC^{\beta}_T$.
\end{theorem}
\begin{proof}
 This theorem can be proved by the consequence of Theorem \ref{thrm4.2} and Theorem \ref{thrm5.1}.
\end{proof}
In the next theorem convergence for the sequence $\{\eta_j^n(s)\}$ is shown.
\begin{theorem}
Let the assumptions $\textbf{(A1)-(A4)}$ satisfied and $y_0,z_0 \in D(\mathcal{A})$. Then,
\begin{equation*}
\lim_{n \rightarrow \infty} \sup_{\{0 \leq s \leq T\}} \bigg( \sum_{j =0}^n \lambda_j^{2 \beta}\mathbb{E}| \eta_j^n(s)-\eta_j(s)|^2 \bigg) = 0.
\end{equation*}
\end{theorem}
\begin{proof}
\begin{align*}
\mathbb{E}\lVert \bar{y}_n(s)-y(s)  \rVert_{\beta}^2 &= \ \mathbb{E} \lVert \mathcal{A}^{\beta}\big[\bar{y}_n(s)-y(s)\big] \rVert^2 \\
 &= \mathbb{E} \lVert \mathcal{A}^{\beta}\bigg[ \sum_{j = 0}^{\infty}\big(\eta_j^n(s)-\eta_j(s)\big) \psi_{j}\bigg] \rVert^2 \\
&= \sum_{j = 0}^{\infty} \lambda_{j}^{2\beta}\mathbb{E}\lVert (\eta_j^n(s)-\eta_j(s)) \psi_{j} \rVert^2.
\end{align*}
Therefore,
\begin{align*}
 \sum_{j = 0}^n \lambda_j^{2 \beta} \lVert \eta_j^n(s)-\eta_j(s) \rVert^2 \leq \ \mathbb{E} \lVert \big[\bar{y}_n(s)-y(s) \big] \rVert^2_{\beta}.
\end{align*}
The requisite result derives from a consequence of Theorem \ref{thrm5.2}.
\end{proof}
\section{Example}\label{Section 5}
This part contains an example that demonstrates the feasibility of our abstract results.\\
We apply the abstract result of this manuscript to a partial differential equation. Let $\mathcal{H} = L^2(0,1)$, we study the following abstract impulsive differential equation of order $\alpha \in (1,2]$: 
\begin{align}\label{5.1}
^c \mathbf{D}_s^{\alpha}\mathcal{V}(s,\xi)  &= \frac{\partial^2}{\partial \xi^2}\mathcal{V}(s,\xi)+ \frac{s}{10(1+s)}\mathcal{V}(s,\xi)+\int_0^s a(s-\varsigma) \frac{e^{-\varsigma}|\mathcal{V}(\varsigma,\xi)|}{3(1+|\mathcal{V}(\varsigma,\xi)|)} dW(\varsigma), \nonumber\\
&\qquad\qquad\qquad\qquad\qquad\qquad\qquad\qquad (s,\xi) \in (\varsigma_k,s_{k+1}] \times (0,1), \nonumber \\
\mathcal{V}(s,\xi) &= \frac{1}{(2k+1)}\sin (ks+\mathcal{V}(s_k^-,\xi)), \ (s,\xi) \in (s_k,\varsigma_k]\times (0,1),\ k = 1,2,\cdots,q, \nonumber  \\
\frac{\partial}{\partial s}\mathcal{V}(s,\xi) &= \big(\frac{k}{2k+1}\big) \cos (ks+\mathcal{V}(s_k^-,\xi)), \ (s,\xi) \in (s_k,\varsigma_k]\times (0,1),\ k = 1,2,\cdots,q, \nonumber \\
\mathcal{V}(s,0) &= \mathcal{V}(s,1) = 0, \quad s \in [0,T], \quad 0 < T < \infty,  \\
\mathcal{V}(0, \xi) &= y_0, \ \frac{\partial}{\partial s}\mathcal{V}(0,\xi) = z_0, \quad \xi \in (0,1), \nonumber
\end{align}
where $y_0, z_0 \in \mathcal{H}$, $W$ is the standard $L^2[0,1]$-valued Weiner process, $a \in L_{loc}^{2p}(0,\infty), 1 \leq p < \infty$ and $0 = \varsigma_0 < s_1 < \varsigma_1 < \cdots < \varsigma_k < s_{k+1} = T$ are some real numbers.\\
We define an operator $\mathcal{A} : D(\mathcal{A}) \subset \mathcal{H} \rightarrow \mathcal{H}$, as follows,
\begin{equation}\label{5.2}
\mathcal{A}\mathcal{V} = \frac{\partial^2 \mathcal{V}}{\partial \xi^2} = \mathcal{V}^{\prime \prime} \ \text{with} \ D(\mathcal{A}) = \{ \mathcal{V} \in \mathcal{H}_0^1(0,1) \cap   \mathcal{H}^2(0,1) : \mathcal{V}^{\prime \prime} \in \mathcal{H} \ \text{and} \ \mathcal{V}(0) = \mathcal{V}(1) = 0\}.
\end{equation} 
Now, let $\beta = \frac{1}{2}$, $D(\mathcal{A}^\frac{1}{2})$ is a Banach space equipped with the norm
\begin{equation}
\parallel y \parallel_{\frac{1}{2}} = \ \parallel \mathcal{A}^{\frac{1}{2}}y \parallel, \quad y \ \in D(\mathcal{A}^\frac{1}{2}),
\end{equation}
and we enlighten this space by $\mathcal{H}_\frac{1}{2} = \mathcal{H}_0^1(0,1).$
For $\lambda \in \mathbb{R}$ and $y \in D(\mathcal{A})$ with $\mathcal{A}y = y^{\prime \prime} = \lambda y,$ we get $\langle \mathcal{A}y,y \rangle = \langle \lambda y, y \rangle$; that is
\begin{equation*}
\langle y^{\prime \prime},y \rangle = -|y^{\prime}|^2_{L^2} = \lambda |y|^2_{L^2}.
\end{equation*}
Therefore, $\lambda < 0$. The general solution of this homogeneous equation is given by
\begin{equation*}
y(\xi) = A \ \text{cos}(\sqrt{\lambda}\xi) + B \ \text{sin}(\sqrt{\lambda}\xi),
\end{equation*}
and the boundary conditions require that $A = 0$ and spectrum $\lambda = \lambda_n = -n^2 \pi^2$, $n \in \mathbb{N}$. Thus, the corresponding solutions (orthonormal eigenfunctions) for each $n \in \mathbb{N}$ are given by
\begin{equation*}
y_n(\xi) = \sqrt{2} \ \text{sin}(n \pi \xi).
\end{equation*}
Since $D(\mathcal{A})$ is a separable Hilbert space therefore there exists a sequence of real numbers $\gamma_{n}$ such that
\begin{equation*}
y(\xi) = \sum_{n = 1}^{\infty} \gamma_n y_n(\xi), \quad \text{ for any} \quad y \in D(\mathcal{A})
\end{equation*}
with
\begin{equation*}
\sum\limits_{n = 1}^{\infty} \gamma_n^2 < \infty \quad \text{and} \quad \sum\limits_{n = 1}^{\infty} \lambda_n^2 \gamma_n^2 < \infty.
\end{equation*}
For $y \in D(\mathcal{A}^\frac{1}{2})$, we have
\begin{equation*}
\mathcal{A}^{\frac{1}{2}}y(\xi) = \sum_{n = 1}^{\infty} \lambda^{\frac{1}{2}}_n \gamma_n y_n(\xi)
\end{equation*}
that is $\sum_{n = 1}^{\infty} \lambda_n \gamma_n^2 < \infty$.\\
Here, $\gamma_n = \langle y,y_n \rangle$. For $y \in D(\mathcal{A})$, operator $\mathcal{A}$ has the following series representation:
\begin{equation}
\mathcal{A}y= -\sum_{n = 1}^{\infty} n^2{\pi}^2 \langle y, y_n \rangle y_n. \nonumber
\end{equation}
Since a strongly continuous cosine family $C(s)_{s \in \mathbb{R}}$ on $\mathcal{H}$ has the infinitesimal generator $\mathcal{A}$. Therefore
\begin{equation}
C(s)y = \sum_{n=1}^{\infty} \cos (n{\pi}s) \langle y,y_n \rangle y_n, \quad y \in \mathcal{H}, \nonumber
\end{equation}
 and the associated sine family $S(s)_{s \in \mathbb{R}}$ on $\mathcal{H}$ is given by
\begin{equation*}
S(s)y = \sum_{n=1}^{\infty} \frac{1}{n{\pi}} \sin (n{\pi}s) \langle y, y_n \rangle y_n, \quad y \in \mathcal{H}.
\end{equation*}
To learn more about cosine family theory, we refer the readers to Fattorini \cite{7}, Travis \& Webb \cite{8}.\\
For $\alpha \in (1,2]$, from the subordinate principal \cite{1}, $\mathcal{A}$ is the infinitesimal generator of fractional cosine family $C_{\alpha}(s)$(strongly continuous and exponentially bounded) such that $C_{\alpha}(0)=I$, and 
\begin{equation}
C_{\alpha}(s) = \int_0^{\infty} \phi_{s,\frac{\alpha}{2}}(t)C(t)dt, \quad s > 0, \nonumber
\end{equation}
where
\begin{equation}
\phi_{s,\frac{\alpha}{2}}(t) = s^{\frac{-\alpha}{2}} \phi_{\frac{\alpha}{2}}(ts^{\frac{-\alpha}{2}}), \nonumber
\end{equation}
and 
\begin{equation}
\phi_{\gamma}(z) = \sum_{n=0}^{\infty} \frac{(-z)^n}{n!\Gamma(-n\gamma-\gamma+1)}, \quad \gamma \in (0,1). \nonumber
\end{equation}
The system (\ref{5.1}) can be transformed into the following abstract FSIDE in $\mathcal{H} = L^2(0,1)$:
\begin{align*}\label{5.4}
^{c}\mathbf{D}_s^{\alpha}y(s) &= \mathcal{A} y(s) + \mathcal{K}(s,y(s)) + \int_0^s a(s-\varsigma)\mathcal{N}(\varsigma,y(\varsigma)) dW(\varsigma), \quad s \in (\varsigma_k, s_{k+1}], \nonumber \\
&\qquad\qquad\qquad\qquad\qquad\qquad\qquad k = 0,1,2, \cdots,q, \nonumber \\
y(s) &= h_k^1(s,y(s_k^-)), \quad s \in (s_k,\varsigma_k], \ k =1,2,\cdots,q, \nonumber \\
y^{\prime}(s) &= h_k^2(s,y(s_k^-)), \quad s \in (s_k,\varsigma_k], \ k =1,2,\cdots,q, \nonumber \\
y(0) &= y_0, \quad y^{\prime}(0) = z_0
\end{align*}
where $\mathcal{V}(s,\cdot)=y(s)$ that is $y(s)(\xi) = \mathcal{V}(s,\xi)$, $(s,\xi) \in [0,T] \times (0,1)$. Eq. (\ref{5.2}) defines the operator $\mathcal{A}$.\\
We define the operators by\\
\vspace{2mm}
(i) $\mathcal{K} : [0,T] \times \mathcal{H}_{\frac{1}{2}} \rightarrow \mathcal{H}$ is given by
\begin{align*}
\mathcal{K}(s,y)(\xi) = \frac{s}{10(1+s)}y(s,\xi)
\end{align*}
(ii) $\mathcal{N} : [0,T] \times \mathcal{H}_{\frac{1}{2}} \rightarrow L_2^0$ is given by
\begin{align*}
\mathcal{N}(s,y)(\xi) = \frac{e^{-s}|y(s,\xi)|}{3(1+|y(s,\xi)|)}
\end{align*}
(iii) $h_k^1 : (s_k,\varsigma_k] \times \mathcal{H}_{\frac{1}{2}} \rightarrow \mathcal{H}_{\frac{1}{2}}, \ k = 1,2,\cdots,q$
\begin{align*}
h_k^1(s,y(s_k^-))(\xi) = \frac{1}{(2k+1)} \sin (ks+y(s_k^-,\xi))
\end{align*}
(iv) $h_k^2 : (s_k,\varsigma_k] \times \mathcal{H}_{\frac{1}{2}} \rightarrow \mathcal{H}_{\frac{1}{2}},\ k = 1,2,\cdots,q$
\begin{align*}
h_k^2(s,y(s_k^-))(\xi) = \frac{k}{(2k+1)} \cos (ks+y(s_k^-,\xi)).
\end{align*}
Now, for each $s \in [0,T]$ and $y_1,y_2 \in \mathcal{H}$, we compute\\
\begin{align}
\mathbb{E}\lVert \mathcal{K}(s,y_1)-\mathcal{K}(s,y_2) \rVert^2 &= \mathbb{E}\bigg[\int_0^1 \bigg |\frac{s}{10(1+s)}y_1(s,\xi)-\frac{s}{10(1+s)}y_2(s,\xi)\bigg |^2 d{\xi}\bigg] \nonumber \\
& \leq \frac{1}{100} \mathbb{E} \bigg[ \int_0^1 \big|y_1(s,\xi)-y_2(s,\xi) \big|^2 d{\xi} \big] \nonumber \\
& \leq \frac{1}{100} \lVert \mathcal{A}^{-\frac{1}{2}} \rVert^2 \mathbb{E} \lVert y_1-y_2 \rVert^2_{\frac{1}{2}} \nonumber
\end{align}
and
\begin{align}
\mathbb{E}\lVert \mathcal{K}(s,y) \rVert^2 &\leq  \theta.\nonumber
\end{align}
Thus the assumption $\textbf{(A2)}$ is satisfied with $L_\mathcal{K} = \frac{1}{100} \lVert \mathcal{A}^{-\frac{1}{2}} \rVert^2$ and $L_\mathcal{K}^{\prime} = \theta$, where $\theta = \frac{1}{100}\lVert \mathcal{A}^{-\frac{1}{2}} \rVert^2\mathbb{E} \lVert \mathcal{A}^{\frac{1}{2}} y \rVert^2$.  Similarly, $L_\mathcal{N}(s) = \frac{e^{-2s}}{9} \lVert \mathcal{A}^{-\frac{1}{2}} \rVert^2$, $L_\mathcal{N}^{\prime}(s) = \frac{e^{-2s}}{9}$ and q = 2. Also, $\textbf{(A4)}$ is satisfied with $ D_{h_k^1} = C_{h_k^1} = \frac{1}{(2k+1)^2}$ and $D_{h_k^2} = C_{h_k^2} = \big(\frac{k}{2k+1}\big)^2$.\\
Next, in particular we consider a case of a control system governed by the following integro-differential equation in the space $\mathcal{H} = \mathbb{R}$.
\begin{align*}
y^{\prime\prime}(s)+a_1y(s)+a_2 \sin s &= \int_0^s a(s-\varsigma)e^{-(s-\varsigma)}g(\varsigma,y(\varsigma)) dW(\varsigma), \  s \in (\varsigma_k,s_{k+1}],\nonumber \\
&\qquad\qquad\qquad\qquad\qquad\qquad\qquad k = 1,2,\cdots,q, \\
y(s) &= a_3 \tanh(y(s_k^-))r(s), \ s \in (s_k,\varsigma_k], k = 1,2,\cdots,q,\\
y^{\prime}(s) &= a_3 \tanh(y(s_k^-))r^{\prime}(s), \ s \in (s_k,\varsigma_k], k = 1,2,\cdots,q,\\
y(0) &= x_0, \quad y^{\prime}(0) = y_0,
\end{align*}
where $a_1 \in \mathbb{R}^+$, $a_2,a_3 \in \mathbb{R}$, $g \in C([0,T] \times \mathbb{R};\mathbb{R})$ and $r \in C([0,T];\mathbb{R})$. We define $\mathcal{A}$ as follows $\mathcal{A}y = -a_1 y$ with $D(\mathcal{A}) = \mathbb{R}$. $C(s) = \cos (\sqrt{a_1}s)$ is the strongly continuous cosine family and the associated sine family is given by $S(s) = \frac{1}{\sqrt{a_1}} \sin (\sqrt{a_1}s)$. The control function is given by $a_2 \sin s$ and the functions $a_3 \tanh(y(s_k^-))r(s)$ and $a_3 \tanh(y(s_k^-))r^{\prime}(s)$ represent non-instantaneous impulses during the interval $(s_k,\varsigma_k],\ k = 1,2,\cdots,q$. 
\section{Conclusion}
In this article, we consider a class of non-instantaneous FSIDE involving Caputo fractional derivative, in an arbitrary separable Hilbert space. We have successfully established the existence, uniqueness and convergence results for the approximate solutions to system (\ref{main1}). To show the existence of solutions to every approximate integral
equation, we use theory of fractional cosine family of linear operators, fixed point technique and stochastic equations in infinite dimensions. The approximate solutions of finite dimension for the given system (\ref{main1}) are constructed by using the projection operator. After that, we prove that a Cauchy sequence is constructed by these approximate solutions, and they converge to the solution of the original problem with respect to a suitable norm. Moreover, we proved some existence and convergence results for such approximations with the help of F-G approximation. At last, an example is given to show the effectiveness of the main theory. Our findings can be extended to investigate the approximation of solution to Hilfer fractional stochastic non-instantaneous impulsive system for future research.

\section*{ORCID}
\emph{Muslim Malik} https://orcid.org/0000-0003-0055-7581 \vspace{2mm}\\
\emph{Shahin Ansari} https://orcid.org/0000-0001-7842-547X

\end{document}